\documentclass[11pt]{amsart}
\usepackage[parfill]{parskip}    
\usepackage{fullpage}
\usepackage{graphicx}
\usepackage{amssymb,amsmath,latexsym}
\usepackage{tikz}
\usetikzlibrary{arrows,decorations.markings,decorations.pathreplacing,calc,matrix,intersections}
\tikzset{->-/.style={decoration={markings,mark=at position #1 with {\arrow{>}}},postaction={decorate}}}

\usepackage{amsmath,amsfonts,amssymb,amscd}
\usepackage{epstopdf}
\usepackage{ifpdf}
\usepackage{color}
\usepackage{float}
\usepackage{amscd,amsmath, mathabx}
\usepackage{amssymb,amsmath,latexsym,color,enumerate,tikz}
\usepackage[all]{xypic}       
\usepackage[varg]{pxfonts}
\usepackage{amsmath,amsthm,amsfonts,amssymb,fancyhdr,graphics,relsize,tikz-cd,mathtools,faktor,tikz,changepage}
\usepackage{graphicx}
\usepackage{placeins}

\usepackage{dsfont}
\definecolor{red}{rgb}{1,0,0} 

 \definecolor{darkgreen}{rgb}{0, .7, 0}

 \definecolor{purple}{rgb}{.7, 0, 1}

\tikzset{mynode/.style={draw,circle,fill=black,inner sep=2pt,outer sep=0.5pt}}

\newtheorem{theorem}{Theorem}[section]
\newtheorem*{theorem*}{Theorem}
\newtheorem*{lemma*}{Lemma}
\newtheorem{proposition}[theorem]{Proposition}
\newtheorem{lemma}[theorem]{Lemma}
\newtheorem{corollary}[theorem]{Corollary}

\theoremstyle{definition}

\theoremstyle{remark}

\usepackage{hyperref}
 
\begin{document}
\title{Higher dimensional algebraic fiberings for pro-$p$ groups}
\author{Dessislava H. Kochloukova}

\address
{Department of Mathematics, State University of Campinas (UNICAMP), 13083-859, Campinas, SP, Brazil }
\email{desi@unicamp.br}


\maketitle

\begin{abstract} We prove some conditions for higher dimensional algebraic fibering of  pro-$p$ group extensions and we establish corollaries about incoherence of pro-$p$ groups.  In particular, if $G = K \rtimes \Gamma$ is a pro-$p$ group, $\Gamma$ a finitely generated free pro-$p$ group with $d(\Gamma) \geq 2$,  $K$ a finitely presented pro-$p$ group with $N$ a normal pro-$p$ subgroup of $K$ such that $K/ N \simeq \mathbb{Z}_p$ and $N$  not finitely generated as a pro-$p$ group, then $G$ is incoherent (in the category of pro-$p$ groups). Furthermore we show that if $K$ is a free pro-$p$ group  with $d(K) = 2$ then either 
$Aut_0(K)$ is incoherent (in the category of pro-$p$ groups) or there is a finitely presented pro-$p$ group, without non-procyclic free pro-$p$ subgroups, that has a metabelian pro-$p$ quotient that is not finitely presented i.e. a pro-$p$ version of a result of Bieri-Strebel does not hold.
\end{abstract}

\section{Introduction}

For a pro-$p$ group $G$ we denote by $K[[G]]$ the completed group algebra of $G$ over the ring $K$, where $K$ is the field with $p$ elements $\mathbb{F}_p$ or the ring of the $p$-adic numbers $\mathbb{Z}_p$. By definition a pro-$p$ group $G$ is of type $FP_m$ if the trivial $\mathbb{Z}_p[[G]]$-module $\mathbb{Z}_p$ has a projective resolution where all projectives in dimension $\leq m$ are finitely generated $\mathbb{Z}_p[[G]]$-modules. Note that $G$ is of type $FP_1$ if and only if $G$ is finitely generated as a pro-$p$ group. And $G$ is of type $FP_2$ if and only if $G$ is finitely presented as a pro-$p$ group i.e. $G \simeq F/ R$, where $F$ is a free pro-$p$ group with a finite free basis $X$ and $R$ is the smallest normal pro-$p$ subgroup of $F$ that contains some fixed finite set of relations of $G$. It is interesting to note that for abstract (discrete) groups the abstract versions of the properties $FP_2$ and finite presentability do not coincide \cite{B-B}.

In this paper  we develop a pro-$p$ version of some of the results on algebraic fibering of abstract group extensions developed  by the author and Vidussi in \cite{K-V} and in the case of results on incoherence we prove results stronger than the ones proved in the abstract case. The results in \cite{K-V} generalise the main results of Friedl and Vidussi in \cite{F-V} and the main results of  Kropholler and Walsh in \cite{K-W}.
The proofs of the results of \cite{F-V},  \cite{K-V} and  \cite{K-W} use the Bieri-Strebal-Neumann-Renz $\Sigma$-invariants introduced in \cite{B-N-S} and \cite{B-R}. 
In \cite{King} King suggested  a  $\Sigma$-invariant in the case of metabelian pro-$p$ groups \cite{King}. We will use the King  invariant
in the proof of Proposition \ref{King-used} but the rest of the results in this paper would have homological proofs independant from the King invariant.

\begin{theorem} \label{Main1} Let $1 \to K \to G \to \Gamma \to 1$ be a short exact sequence of pro-$p$ groups such that $G$ and $K$ are of type $FP_{n_0}$, $\Gamma^{ab}$ is infinite and  there is a normal pro-$p$ subgroup $N$ of $K$ such that $G'\cap K \subseteq N$, $K/ N \simeq \mathbb{Z}_p$ and $N$ is of type $FP_{n_0-1}$. Then there is a normal pro-$p$ subgroup $M$ of $G$ such that $G/ M \simeq \mathbb{Z}_p$, $M \cap K = N$ and $M$ is of type $FP_{n_0}$. Furthermore if $K$, $G$ and $N$ are of type $FP_{\infty}$ then $M$ can be chosen of type $FP_{\infty}$.
\end{theorem}

We call a discrete pro-$p$ character of $G$ a non-trivial homomorphism of pro-$p$ groups  $\alpha : G \to H$  such that $H \simeq \mathbb{Z}_p$.   Then the theorem could be restated as : assume that  $G$ and $K$ are of type $FP_{n_0}$, $\Gamma^{ab}$ is infinite and there is  a discrete pro-$p$ character $\alpha$ of $G$ such that $\alpha | _K \not= 0$, $Ker(\alpha) \cap K = N$  is of type $FP_{n_0-1}$. Then there exists 
a discrete pro-$p$ character $\mu$ of $G$ such that $M = Ker(\mu)$ is of type $FP_{n_0}$ and $\mu |_K = \alpha |_K$, in particular $M \cap K = N$.

\medskip

There is a lot in the literature on coherent abstract groups but very little is known for coherent pro-$p$ groups. Similar to the abstract case a pro-$p$ group $G$ is coherent (in the category of pro-$p$ groups) if every finitely generated pro-$p$ subgroup of $G$ is finitely presented as a pro-$p$ group i.e. is of type $FP_2$. We generalise this concept  and define that a pro-$p$ group $G$ is $n$-coherent if any pro-$p$ subgroup of $G$ that is of  type $FP_n$ is  of type $FP_{n+1}$.  Thus a pro-$p$ group is 1-coherent if and only if  it is coherent (in the category of pro-$p$ groups).

\begin{corollary} \label{cor1}  Let $K$, $\Gamma$ and $G = K \rtimes \Gamma$ be pro-$p$ groups, where $\Gamma$ is finitely generated free pro-$p$ but not pro-$p$ cyclic. Suppose that $K$ is of type $FP_{n_0 + 1}$ and there is a normal pro-$p$ subgroup $N$ of $K$ such that  $G'\cap K \subseteq N$, $K/ N \simeq \mathbb{Z}_p$ and $N$ is of type $FP_{n_0 -1}$ but is not of type $FP_{n_0}$.  Then there is a normal pro-$p$ subgroup $M$ of $G$ such that $G/ M \simeq \mathbb{Z}_p$, $M \cap K = N$ and $M$ is of type $FP_{n_0}$ but is not of type $FP_{n_0 + 1}$. In particular $G$ is not $n_0$-coherent.
\end{corollary}

As in the case of Theorem \ref{Main1}, Corollary \ref{cor1} can be restated in terms of discrete pro-$p$ characters. 
 
 For a free abstract  group its rank is  the minimal number of generators. 
Since for pro-$p$ groups $G$ finite rank is used for a notion different from the one addopted for abstract groups, we write $d(G)$ for the minimal number of generators of $G$.
It is known that abstract (free finite rank)-by-$\mathbb{Z}$ groups are coherent \cite{F-H}.    There is a conjecture suggested by Wise and independently by Kropholler and Walsh that an abstract (free of finite rank $\geq 2$)-by-(free of finite rank $\geq 2$) group is incoherent, see \cite{K-W}.
In \cite{K-W} Kropholler and Walsh proved that (free of rank 2)-by-(free of finite rank $\geq 2$) abstract group is incoherent. The proof uses significantly that for a free abstract group $F_2$ of rank 2 we have that $Out(F_2) \simeq GL_2(\mathbb{Z})$ and some explicit calculations with a finite generating set of a subgroup of finite index in $GL_2(\mathbb{Z})$ were used. Such an approach would not work for pro-$p$ groups since by Romankov's result in \cite{R} the authomorphism group of a free pro-$p$ group $G$, where $2 \leq d(G) < \infty$, is not finitely generated as a topological group.
Still a pro-$p$ version of the Kropholler-Walsh result holds and it is a particular case of Corollary  \ref{CorMain*} that follows from the following quite general theorem.

\begin{theorem} \label{ThmMain2} 
Let $G = K \rtimes \Gamma$ be a pro-$p$ group with $K$ a finitely presented pro-$p$ group  such that there is a normal pro-$p$ subgroup $N$ of $K$ such that $K/ N \simeq \mathbb{Z}_p$ and $N$ is not finitely generated, $\Gamma$ a finitely generated free pro-$p$ group with $d(\Gamma) \geq 2$.
Then $G$ is incoherent (in the category of pro-$p$ groups).
\end{theorem}

 The class  of pro-$p$ groups $\mathcal{L}$ was first considered by the author and Zalesskii in \cite{K-Z}.
 This class of groups contains all finitely generated free pro-$p$ groups. It shares many properties with the class of abstract limit groups and is defined using extensions of centralizers. There are many open questions about the class of pro-$p$ groups  $\mathcal{L}$. For example, by Wilton's result from \cite{W} every finitely generated subgroup of an abstract limit group is a virtual retract, but the pro-$p$ version of this result is still an open problem. In order to prove Corollary \ref{CorMain*} we show in Proposition \ref{abel} that the abelianization of any non-trivial pro-$p$ group from $\mathcal{L}$ is always infinite.
The same argument can be addapted for the class of abstract limit groups.

\begin{corollary} \label{CorMain*}
Let $G = K \rtimes \Gamma$ be a pro-$p$ group with $K$ a non-abelian pro-$p$ group from the class $\mathcal{L}$, $\Gamma$ a finitely generated free pro-$p$ group with $d(\Gamma) \geq 2$.
Then $G$ is incoherent (in the category of pro-$p$ groups). In particular if $K$ is a finitely generated free pro-$p$ group  with $d(K) \geq 2$ then  $G$ is incoherent (in the category of pro-$p$ groups).
\end{corollary}

For a finite rank free pro-$p$ group $F$ the structure of $Aut(F)$ was studied first by Lubotsky in \cite{L}. $Aut(F)$ is a topological group with a pro-$p$ subgroup of finite index. In \cite{G} Gordon proved that the automorphism group of an abstract free group of rank 2 is incoherent. Unfortunately we could not prove  a pro-$p$ version of this result but still it would hold if  the group of outer pro-$p$ automophisms of a free pro-$p$ group of rank 2 contains a free non-procyclic pro-$p$ subgroup. For a free abstract group $F_2$ of rank 2 we have that $Out(F_2) \simeq GL_2(\mathbb{Z})$ and since $SL_2(\mathbb{Z})$ is isomorphic to the free amalgamated product of $C_4$ and $C_6$ over a copy of $C_2$ it follows easily that $GL_2(\mathbb{Z})$ contais a free non-cyclic abstract group (or use the Tits alternative), hence $Out(F_2)$ contains a free non-cyclic abstract group.  Nevertheless the group $GL_2^1(\mathbb{Z}_p) = Ker (GL_2(\mathbb{Z}_p) \to GL_2(\mathbb{F}_p))$ does not contain a free pro-$p$ non-procyclic pro-$p$ subgroup, since it is $p$-adic analytic and so there is an upper limit on the number of generators of finitely generated pro-$p$ subgroups \cite{DSMS}. For related results on non-existance of free pro-$p$ subgroups in matrix groups see \cite{B-L},  \cite{BE-Z}, \cite{Z}.

Let $G$ be a finitely generated pro-$p$ group. Define $Aut_0(G) = Ker (Aut(G) \to Aut(G/ G^*))$, where $G^*$ is the Frattini subgroup of $G$. Then $Aut_0(G)$ is a pro-$p$ subgroup of $Aut(G)$ of finite index.

\begin{corollary} \label{aut}
Suppose that $K$ is a free pro-$p$ group with $d(K) = 2$. If $Out(K)$ contains  a pro-$p$  free non-procyclic subgroup then 
$Aut_0(K)$ is incoherent (in the category of pro-$p$ groups).
\end{corollary}

By the Bieri-Strebel results in \cite{B-S} for a finitely presented abstract group $H$ that does not contain free non-cyclic abstract subgroups, every metabelian quotient of $H$ is finitely presented. It is an open question whether a pro-$p$ version of the Bieri-Strebel result  holds i.e. whether if $G$ is a finitely presented pro-$p$ group without free non-procyclic pro-$p$ subgroups then every metabelian pro-$p$ quotient of $G$ is finitely presented as a pro-$p$ group. Note that by the King classfication of the finitely presented metabelian pro-$p$ groups in \cite{King2} every pro-$p$ quotient of a finitely presented metabelian pro-$p$ group is finitely presented pro-$p$.  Using Corollary \ref{aut} and some ideas introduced by Romankov in  \cite{R0}, \cite{R} we prove the following result.

\begin{corollary} \label{alternative}
Suppose that $K$ is a free pro-$p$ group  with $d(K) = 2$.  Then either 
$Aut_0(K)$ is incoherent (in the category of pro-$p$ groups) or the pro-$p$ version of the Bieri-Strebel result does not hold.
\end{corollary}

{\bf Acknowledgments} The author was partially supported by Bolsa de produtividade em pesquisa CNPq 305457/2021-7 and Projeto tem\'atico FAPESP 18/23690-6.
 
\section{Preliminaries}

\subsection{Homological finiteness properties for pro-$p$ groups} \label{homological-pro-p}
Let $G$ be a pro-$p$ group.
By definition $$\mathbb{Z}_p[[G]] = {\varprojlim} \frac{\mathbb{Z}}{p^i \mathbb{Z}} [[G/ U]],$$ where the inverse limit is over all $i \geq 1$ and $U$ open subgroups of $G$. And $$\mathbb{F}_p[[G]] =\mathbb{Z}_p[[G]]/ p \mathbb{Z}_p[[G]] =  {\varprojlim} \mathbb{F}_p[[G/U]]$$
where the inverse limit is over all open subgroups $U$ of $G$.

 By definition $G$ is of type $FP_m$ if the trivial $\mathbb{Z}_p[[G]]$-module $\mathbb{Z}_p$ has a projective resolution where all projectives in dimension $\leq m$ are finitely generated $\mathbb{Z}_p[[G]]$-modules.
By \cite{King} for a pro-$p$ group the following conditions are equivalent : 

1) $G$ is of type $FP_m$;

2) $H_i(G, \mathbb{Z}_p)$ is a finitely generated (abelian) pro-$p$ group for $i \leq m$;

3) $H_i(G, \mathbb{F}_p)$ is finite for $i \leq m$;

4) for $K$ either $\mathbb{F}_p$ or $\mathbb{Z}_p$ and $N$  a normal pro-$p$ subgroup of $G$ such that $K[[G]]$ is left and right Noetherian the homology groups $H_i(N, K)$ are finitely generated as $K[[G/N]]$-modules for $i \leq m$, where the $G/N$ action is induced by the conjugation action of $G$ on $N$.

 The equivalence of the above conditions is a corollary of the fact that $\mathbb{Z}_p[[G]]$ and $\mathbb{F}_p[[G]]$ are local rings.
Here $H_i(N, \mathbb{Z}_p)$ and  $H_i(N, \mathbb{F}_p)$ are the standard homology groups of pro-$p$ groups with coeficients in the trivial  pro-$p$ $\mathbb{Z}_p[[G]]$-modules $\mathbb{Z}_p$ and $\mathbb{F}_p$, for more on homology groups see \cite{R-Z}.

\subsection{The King invariant}

Let $Q$ be a finitely generated abelian pro-$p$ group and
let $\mathbb{F}$ be the algebraic closure of $\mathbb{F}_p$. Denote by $\mathbb{F}[[t]]^{\times}$ the multiplicative group of invertible elements in $\mathbb{F} [[t]]$. Consider
$$T(Q) = \{ \chi : Q \to \mathbb{F}[[t]]^{\times} \ | \ \chi \hbox{ is a continuous homomorphism} \},$$
where $\mathbb{F}[[t]]^{\times} $ is a topological group with topology induced by the topology of the ring $\mathbb{F}[[t]]$, given by the sequence of ideals $(t) \supseteq (t^2) \supseteq \ldots \supseteq (t^i) \supseteq \ldots $. Note that since $\chi$ is continuous we have that $\chi(Q) \subset 1 + t \mathbb{F}[[t]]$.

For $\chi \in T(Q)$ there is a unique continious ring homomorphism
$$\overline{\chi} : \mathbb{Z}_p[[Q]] \to \mathbb{F}[[t]]$$
that extends $\chi$.

Let $A$ be a finitely generated pro-$p$ $\mathbb{Z}_p[[Q]]$-module. In \cite{King2} King defined
the following invariant
$$ \Delta(A) = \{  \chi \in T(Q) \ | \ ann_{\mathbb{Z}_p[[Q]]}(A) \subseteq Ker (\overline{\chi})  \}. $$
In \cite{King2} King used the notation $\Xi(A)$, that we here substitute by $\Delta(A)$.

Let $P$ be a pro-$p$ subgroup of $Q$. Define $T(Q, P) = \{ \chi \in T(Q) \ | \ \chi(P) = 1 \}$.

\begin{theorem} \cite[Thm B]{King2},  \cite[Lemma 2.5]{King2} \label{fin-gen-module}  Let $Q$ be a finitely generated abelian pro-$p$ group.
Let $A$ be a finitely generated pro-$p$ $\mathbb{Z}_p[[Q]]$-module.

a) Then $A$ is finitely generated as an abelian pro-$p$ group if and only if $\Delta(A) = \{ 1 \}$.

b) If $P$ is  a pro-$p$ subgroup of $Q$ then $T(Q, P) \cap \Delta(A) = \Delta(A/ [A, P]).$ In particular $A$ is finitely generated as a pro-$p$ $\mathbb{Z}_p[[P]]$-module if and only if $T(Q, P) \cap \Delta(A) = \{1 \}$.
\end{theorem}

We state  the classification of the finitely presented metabelian pro-$p$ groups given by King in \cite{King2}.

\begin{theorem} \cite{King2} \label{KK} Let $1 \to A \to G \to Q \to 1$ be a short exact sequence of  pro-$p$ groups, where $G$ is a finitely generated pro-$p$ group and $A$ and $Q$ are abelian pro-$p$ groups. Then $G$ is a finitely presented pro-$p$ group if and only if $\Delta(A) \cap \Delta(A)^{-1} = \{ 1 \}$.
\end{theorem}

{\bf Example} Let $A = \mathbb{F}_p[[s]]$,  $Q = \mathbb{Z}_p $, $G = A \rtimes Q$, where $\mathbb{Z}_p$ has a generator $b$ and $b$ acts via conjugation on $A$ by multiplication with $1 + s$. Since $ann_{\mathbb{Z}_p[[Q]]} (A) = p \mathbb{Z}_p[[Q]] \subseteq Ker(\overline{\chi})$ for any $\chi \in T(Q)$, we conclude that $\Delta(A) = T(Q)  = \Delta(A) ^{-1}$. Hence by Theorem \ref{KK} $G$ is not finitely presented.

\section{Proofs}

We start by citing a result on abstract groups. We recall first that an abstract group $G$ is of type $FP_m$ if the trivial $\mathbb{Z} G$-module $\mathbb{Z}$ has a projective resolution where all projectives in dimension $\leq m$ are finitely generated. An abstract group $G$ is of (homotopical) type $F_n$ if there is a classifying space $K(G, 1)$ with finite $n$-skeleton. If $n \geq 2$ then $G$ is of type $F_n$ if and only if it is of type $FP_n$ and is finitely presented ( as an abstract group).
 The homotopical part of Proposition \ref{Kuck} was proved by Kuckuck in \cite{K} and the homological part of Proposition \ref{Kuck} was proved by    the author and Lima in \cite{K-L}. The former has a geometric proof and the latter an algebraic one. 

\begin{proposition} \cite{K}, \cite{K-L} \label{Kuck}  Let $n \geq 1$ be a natural number, $A \hookrightarrow B \twoheadrightarrow C$ a short exact sequence of groups with  $A$ of type $F_n$ ( resp. of type $FP_n$) and $C$ of type $F_{n+1}$ (resp. of type $FP_{n+1}$). Assume there is another short exact sequence of groups $A \hookrightarrow B_0 \twoheadrightarrow C_0$ with $B_0$ of type $F_{n+1}$ (resp. of type $FP_{n+1}$) and that there is a group homomorphism $\theta: B_0 \rightarrow B$ such that $\theta|_A = id_A$,  i.e. there is a commutative diagram of homomorphisms of groups  $$\xymatrix{A \ \ar@{^{(}->}[r] \ar[d]_{id_A} & B_0 \ar@{->>}[r]^{\pi_0} \ar[d]_{\theta} & C_0 \ar@{.>}[d]^{\nu} \\ A \ \ar@{^{(}->}[r] & B \ar@{->>}[r]^{\pi} & C}$$  Then $B$ is of type $F_{n+1}$ ( resp. of type $FP_{n+1}$).
\end{proposition}

We prove a pro-$p$ version of the above proposition. Recall that the property $FP_m$ for pro-$p$ groups was discussed in Section \ref{homological-pro-p}.

\begin{lemma} \label{L1}  Let $n \geq 1$ be a natural number, $A \hookrightarrow B \twoheadrightarrow C$ a short exact sequence of pro-$p$ groups with  $A$ of type $FP_n$ and $C$ of type $FP_{n+1}$. Assume there is another short exact sequence of pro-$p$ groups $A \hookrightarrow B_0 \twoheadrightarrow C_0$ with $B_0$ of type $FP_{n+1}$ and that there is a homomorphism of pro-$p$ groups $\theta: B_0 \rightarrow B$ such that $\theta|_A = id_A$,  i.e. there is a commutative diagram of homomorphisms of pro-$p$ groups  $$\xymatrix{A \ \ar@{^{(}->}[r] \ar[d]_{id_A} & B_0 \ar@{->>}[r]^{\pi_0} \ar[d]_{\theta} & C_0 \ar@{.>}[d]^{\nu} \\ A \ \ar@{^{(}->}[r] & B \ar@{->>}[r]^{\pi} & C}$$  Then $B$ is of type $FP_{n+1}$.
\end{lemma}

\begin{proof}
Consider the LHS-spectral sequence $$E_{i,j}^2 = H_i(C_0, H_j(A, \mathbb{F}_p))$$ that converges to $H_{i+ j} (B_0, \mathbb{F}_p)$. Similarly there is the LHS spectral sequence
 $$\widehat{E}_{i,j}^2 = H_i(C, H_j(A, \mathbb{F}_p))$$ that converges to $H_{i+ j} (B, \mathbb{F}_p)$.
 
 Since $A$ is of type $FP_n$ we have that $H_j(A, \mathbb{F}_p)$ is finite for all $j \leq n$. Then there is a pro-$p$ subgroup $C_1$ of finite index in $C$ such that $C_1$ acts trivially on $H_j(A, \mathbb{F}_p)$  for every $j \leq n$. Since $C$ is of type $FP_{n+1}$ we have that $C_1$ is of type $FP_{n+1}$. Then $$H_i(C_1, H_j(A, \mathbb{F}_p)) \simeq \oplus H_i(C_1, \mathbb{F}_p)\hbox{ is finite for }j \leq n, i \leq n+1,$$ where we have $dim_{\mathbb{F}_p}  H_j(A, \mathbb{F}_p)$ direct summands. Since $C_1$ has finite index in $C$ we deduce that $$\widehat{E}_{i,j}^2 =  H_i(C, H_j(A, \mathbb{F}_p))\hbox{ is finite  for }j \leq n, i \leq n+1,$$ hence by the convergence of the second spectral sequence we obtain that $$H_k(B, \mathbb{F}_p) \hbox{ is finite for }k \leq n.$$ Note that we have shown that if $i + j = n+1 , i \not= 0$ then $\widehat{E}_{i,j}^2$ is finite, hence $\widehat{E}_{i,j}^{\infty}$ is finite.  By the convergence of the spectral sequence there is a filtration of $H_{n+1}(B, \mathbb{F}_p)$
 $$
 0 = F_{-1}(H_{n+1}(B, \mathbb{F}_p)) \subseteq  \ldots \subseteq  F_{i}(H_{n+1}(B, \mathbb{F}_p)) \subseteq  F_{i+1}(H_{n+1}(B, \mathbb{F}_p))$$ $$ \subseteq \ldots \subseteq  F_{n+1}(H_{n+1}(B, \mathbb{F}_p)) = H_{n+1}(B, \mathbb{F}_p)
 $$
 where $ F_{i}(H_{n+1}(B, \mathbb{F}_p)) /  F_{i-1}(H_{n+1}(B, \mathbb{F}_p)) \simeq  \widehat{E}_{i, n+1 - i}^{ \infty}$.
 Thus $$H_{n+1}(B, \mathbb{F}_p) \hbox{ is finite if and only if  }\widehat{E}_{0,n+1}^{\infty} \hbox{ is finite.}$$ Note that since any differential that comes out from $\widehat{E}^r_{0, n+1}$ is zero  we have that $\widehat{E}_{0,n+1}^{\infty}$ is a quotient of $\widehat{E}^2_{0, n+1} = H_0(C, H_{n+1}(A, \mathbb{F}_p))$, thus there is a map
 $$\mu : H_0(C, H_{n+1}(A, \mathbb{F}_p)) \to H_{n+1}(B, \mathbb{F}_p)$$ 
 with image that equals $\widehat{E}^{ \infty}_{0, n+1}$. Thus $B$ is of type $FP_{n+1}$ if and only if $Im (\mu)$ is finite.
 
 Similarly there is a map
 $$\mu_0 : H_0(C_0, H_{n+1}(A, \mathbb{F}_p)) \to H_{n+1}(B_0, \mathbb{F}_p)$$ 
 with image that equals ${E}^{ \infty}_{0, n+1}$ and such that $B_0$ is of type $FP_{n+1}$ if and only if $Im (\mu_0)$ is finite. Since $B_0$ is of type $FP_{n+1}$  we conclude that $Im (\mu_0)$ is finite. 
 
 The naturallity of the LHS spectral sequence implies that we have the commutative diagram 
$$
\begin{CD}
H_0(C_0, H_{n+1}(A, \mathbb{F}_p)) @>\rho >> H_0(C, H_{n+1}(A, \mathbb{F}_p))\\
@V\mu_0 VV @VV\mu V\\
H_{n+1}(B_0, \mathbb{F}_p) @> \rho_0 >> H_{n+1}(B, \mathbb{F}_p)
\end{CD}
$$
 where the maps $\rho$ and $\rho_0$ are induced by $\nu$.
 Recall that the action of $B_0$ on $A$ via conjugation induces an action of $B_0$ on $H_{n+1}(A, \mathbb{F}_p)$ where $A$ acts trivially and this induces the action of $C_0$ on $H_{n+1}(A, \mathbb{F}_p)$ that is used to define $H_0(C_0, H_{n+1}(A, \mathbb{F}_p))$. Similarly the action of $B$ on $A$ via conjugation induces an action of $B$ on $H_{n+1}(A, \mathbb{F}_p)$ where $A$ acts trivially and this induces the action of $C$ on $H_{n+1}(A, \mathbb{F}_p)$ that is used to define $H_0(C, H_{n+1}(A, \mathbb{F}_p))$.  Recall that the map $$\rho : H_0(C_0, H_{n+1}(A, \mathbb{F}_p)) \to H_0(C, H_{n+1}(A, \mathbb{F}_p))$$ from the commutative diagram is induced by $\nu$. If $\nu$ is surjective then $\rho$ is an isomorphism; if $\nu$ is injective then $\rho$ is surjective. Since every homomorphsim $\nu$ is composition of one epimorphsim followed by one monomorphism we conclude that $\rho$ is always surjective. Then 
 $$Im (\mu) = Im (\mu \circ \rho) = Im (\rho_0 \circ \mu_0) \hbox{ is a quotient of } Im (\mu_0)$$
 Since $Im (\mu_0)$ is finite we conclude that $Im (\mu)$ is finite. Hence $B$ is of type $FP_{n+1}$ as required.
\end{proof}

Recall that a pro-$p$ HNN extension is called proper if the canonical map from the base group to  the pro-$p$ HNN extension is injective.
\begin{lemma} \label{HNN}
Let $G = \langle A, t \ | \ K^t = K \rangle$ be a proper pro-$p$ HNN extension. Suppose that $A$, $K$ are pro-$p$ groups of type $FP_m$ and $M$ is a normal pro-$p$ subgroup of $G$ such that $G/ M \simeq \mathbb{Z}_p$, $K \not\subseteq M$ and $M \cap A$ is of type $FP_m$. Then the following holds: 

a)  $M$ is of type $FP_{m}$ if and only if $M \cap K$ is of type $FP_{m-1}$;

b) if $M$ is of type $FP_{m+1}$ then $M \cap K$ is of type $FP_m$.
\end{lemma}

\begin{proof}
The proper pro-$p$ HNN extension gives rize to the exact sequence of $\mathbb{F}_p[[G]]$-modules
\begin{equation} \label{eq} 0 \to \mathbb{F}_p[[G]] \otimes_{\mathbb{F}_p[[K]]} \mathbb{F}_p \to  
\mathbb{F}_p[[G]] \otimes_{\mathbb{F}_p[[A]]} \mathbb{F}_p \to \mathbb{F}_p \to 0
\end{equation}
Note that since $K \not\subseteq M$ we have that $M \setminus G / K = G/ MK$ is a proper pro-$p$ quotient of $G/ M \simeq \mathbb{Z}_p$, hence is finite.
Similarly  $M \setminus G / A = G/ MA$ is finite. Note that there is an isomorphism of (left) $\mathbb{F}_p[[M]]$-modules 
$$\mathbb{F}_p[[G]] \otimes_{\mathbb{F}_p[[K]]} \mathbb{F}_p \simeq (\oplus_{t \in M \setminus  G / K} \mathbb{F}_p[[M]] t \mathbb{F}_p[[K]]) \otimes_{\mathbb{F}_p[[K]]} \mathbb{F}_p \simeq \oplus_{t \in M \setminus  G / K} \mathbb{F}_p[[M]] \otimes_{\mathbb{F}_p[[ M \cap t K t^{ -1}]]} \mathbb{F}_p
$$
Similarly 
there is an isomorphism of (left) $\mathbb{F}_p[[M]]$-modules 
$$\mathbb{F}_p[[G]] \otimes_{\mathbb{F}_p[[A]]} \mathbb{F}_p \simeq (\oplus_{t \in M \setminus  G / A} \mathbb{F}_p[[M]] t \mathbb{F}_p[[A]]) \otimes_{\mathbb{F}_p[[A]]} \mathbb{F}_p \simeq \oplus_{t \in M \setminus  G / A} \mathbb{F}_p[[M]] \otimes_{\mathbb{F}_p[[M \cap t A t^{ -1}]]} \mathbb{F}_p
$$
The short exact sequence (\ref{eq}) gives rise to a long exact sequence in pro-$p$ homology
$$ \ldots \to H_{m+1}(M, \mathbb{F}_p) \to H_m(M, \mathbb{F}_p[[G]] \otimes_{\mathbb{F}_p[[K]]} \mathbb{F}_p ) \to H_m(M, \mathbb{F}_p[[G]] \otimes_{\mathbb{F}_p[[A]]} \mathbb{F}_p) \to H_m (M, \mathbb{F}_p) $$ $$ \to H_{m-1}(M, \mathbb{F}_p[[G]] \otimes_{\mathbb{F}_p[[K]]} \mathbb{F}_p ) \to \ldots
\to  H_1(M, \mathbb{F}_p[[G]] \otimes_{\mathbb{F}_p[[A]]} \mathbb{F}_p) \to H_1(M, \mathbb{F}_p) \to$$ 
$$H_0(M,  \mathbb{F}_p[[G]] \otimes_{\mathbb{F}_p[[K]]} \mathbb{F}_p )  \to H_0(M,  \mathbb{F}_p[[G]] \otimes_{\mathbb{F}_p[[A]]} \mathbb{F}_p ) \to H_0(M, \mathbb{F}_p) \to 0.
$$
Note that
$$H_i(M, \mathbb{F}_p[[G]] \otimes_{\mathbb{F}_p[[K]]} \mathbb{F}_p ) \simeq
H_i (M,  \oplus_{t \in M \setminus  G / K} \mathbb{F}_p[[M]] \otimes_{\mathbb{F}_p[[M \cap t K t^{ -1}]]} \mathbb{F}_p) \simeq$$
$$ \oplus_{t \in M \setminus  G / K} H_i (M,  \mathbb{F}_p[[M]] \otimes_{\mathbb{F}_p[[M \cap t K t^{ -1}]]} \mathbb{F}_p) \simeq
\oplus_{t \in M \setminus  G / K} H_i (M \cap t K t^{ -1},   \mathbb{F}_p) =$$ $$ \oplus_{t \in M \setminus  G / K} H_i (t( M \cap K) t^{ -1},   \mathbb{F}_p) \simeq 
\oplus_{t \in M \setminus  G / K} H_i (M \cap K,   \mathbb{F}_p).
$$
Similarly 
$$H_i(M, \mathbb{F}_p[[G]] \otimes_{\mathbb{F}_p[[A]]} \mathbb{F}_p ) \simeq
\oplus_{t \in M \setminus  G / A} H_i (M \cap A,   \mathbb{F}_p).
$$
Then the long exact sequence could be rewritten as
$$ \ldots \to H_{m+1}(M, \mathbb{F}_p) \to \oplus_{t \in M \setminus  G / K} H_m(M \cap K, \mathbb{F}_p)\to \oplus_{t \in M \setminus  G / A} H_m(M \cap A, \mathbb{F}_p)
 \to H_m (M, \mathbb{F}_p) $$ $$ \to \oplus_{t \in M \setminus  G / K} H_{m-1}(M \cap K, \mathbb{F}_p) \to \ldots
\to \oplus_{t \in M \setminus  G / A} H_1(M \cap A, \mathbb{F}_p) \to H_1(M, \mathbb{F}_p) \to$$ 
$$  \oplus_{t \in M \setminus  G / K} H_0(M \cap K, \mathbb{F}_p)   \to \oplus_{t \in M \setminus  G / A} H_0(M \cap A, \mathbb{F}_p) \to H_0(M, \mathbb{F}_p) \to 0.
$$
Since $M \cap A$ is of type $FP_m$ we have that $H_i(M \cap A, \mathbb{F}_p)$ is finite for $i \leq m$. Combining with  $M \setminus G / A$ is finite, we conclude that $ \oplus_{t \in M \setminus  G / A} H_i(M \cap A, \mathbb{F}_p)$ is finite for $i \leq m$.

a) Note that $M$ is of type $FP_m$ if and only if $H_i(M, \mathbb{F}_p)$ is finite for $i \leq m$.  By the above long exact sequence together with the fact that $M \setminus G / K$ is finite,  $H_i(M, \mathbb{F}_p)$ is finite for $i \leq m$ if and only if
 $ \oplus_{t \in M \setminus  G / K} H_i(M \cap K, \mathbb{F}_p)$ is finite for $i \leq m-1$ i.e.    $M \cap K$ is of type $FP_{m-1}$.

b) If $M$ is of type $FP_{m+1}$ then $H_{m+1}(M, \mathbb{F}_p)$ is finite and since $H_m(M \cap A, \mathbb{Z}_p)$ is finite by the long exact sequence $H_m(M \cap K, \mathbb{F}_p)$ is finite. We already know by a) that $M \cap K$ is of type $FP_{m-1}$, hence $M \cap K$ is of type $FP_m$.
\end{proof}

For a pro-$p$ group $G$ with a subset $S$ denote by $\langle S \rangle$ the pro-$p$ subgroup of $G$ generated by $S$.

\begin{proposition} \label{King-used} Let $Q \simeq \mathbb{Z}_p^2 = \langle x, y \rangle$ and $A$ be a finitely generated pro-$p$ $\mathbb{Z}_p[[Q]]$-module. Suppose that for $H = \langle x \rangle$ we have that $A$ is finitely generated as a pro-$p$ $\mathbb{Z}_p[[H]]$-module. Let $H_j = \langle x y^{ - p^j} \rangle$. Then there is $j_0 > 0$ such that for every $j \geq j_0$ we have that $A$ is finitely generated as $\mathbb{Z}_p[[H_j]]$-module. 
\end{proposition}

\begin{proof}
By Theorem \ref{fin-gen-module} if $P$ is  a pro-$p$ subgroup of $Q$ then $A$ is finitely generated as $\mathbb{Z}_p[[P]]$-module if and only if $T(Q, P) \cap \Delta(A) = \{1 \}$.
Let $$J = ann_{\mathbb{Z}_p[[Q]]}(A).$$ Since $A$ is finitely generated as a pro-$p$ $\mathbb{Z}_p[[H]]$-module for every $\chi \in T(Q, H) \setminus \{ 1 \}$ we have that $J \not\subseteq Ker (\overline{\chi})$.

Let $\mu_j \in T(Q, H_j) \setminus \{ 1 \}$. We aim to show that for sufficiently big $j$ we have that $\mu_j \not\in \Delta(A)$. Then by Theorem \ref{fin-gen-module}, $A$ is finitely generated as $\mathbb{Z}_p[[H_j]]$-module. 

Let $$\overline{\mu}_j : \mathbb{Z}_p[[Q]] \to \mathbb{F}[[t]]$$ be the continious ring homomorphism induced by $\mu_j$. Since $\overline{\mu}_j(H_j) = 1$ we have
$$\mu_j(x) = \mu_j(y^{p^j}).$$

Let $\chi \in T(Q, H)  \setminus \{ 1 \}$ be such that $$\chi(y) = \mu_j(y)$$ and
$$\overline{\chi} : \mathbb{Z}_p[[Q]] \to \mathbb{F}[[t]]$$ be the continious ring homomorphism induced by $\chi$. Recall that  $\chi \in T(Q, H) $ implies that $\chi(x) = 1$. Then there is $\lambda \in J$ such that $\overline{\chi} (\lambda) \not= 0$. Note that $\lambda \in \mathbb{Z}_p[[Q]] = \mathbb{Z}_p[[ t_1, t_2]]$, where $x = 1 + t_1, y = 1 + t_2$ and since  $\chi(y) = \mu_j(y)$ we have
$$0 \not= \overline{\chi} (\lambda) = \overline{\chi} (\lambda |_{t_1 = 0}) = \overline{\mu}_j(\lambda |_{t_1 =0}).$$
Note that $$ \overline{\mu}_j(t_2) =  \overline{\mu}_j(1 + t_2) -  \overline{\mu}_j(1)   \in 1 + t \mathbb{F}[[t]] - 1 = t \mathbb{F}[[t]]$$ hence
$ \overline{\mu}_j(t_2)^{ p^j} \in t^{ p^j} \mathbb{F}[[t]]$. This together with the condition $\mu_j(x) = \mu_j(y^{p^j})$  implies
$$\overline{\mu}_j(\lambda) = \overline{\mu}_j(\lambda |_{t_1 = t_2^{p^j}}) \in \overline{\mu}_j(\lambda |_{t_1 =0}) + 
 t^{ p^j} \mathbb{F}[[t]].$$
 Suppose that
$$0 \not= \overline{\mu}_j(\lambda |_{t_1 =0}) \in f t^m + t^{ m+1} \mathbb{F}[[t]]$$
where $f \in \mathbb{F} \setminus \{ 0 \}, m \geq 0$. Then choose $j_0> 0$ such that $p^{j_0} > m$ and this implies that for $j \geq j_0$ we have $\overline{\mu}_j(\lambda) \not= 0$. Hence $\mu_j \notin \Delta(A)$
\end{proof}

\begin{proposition} \label{quotientZp}
Let $G$ be a pro-$p$ group with a normal pro-$p$ subgroup $G_0$ such that $G/ G_0 \simeq \mathbb{Z}_p^2$. Let $S$ be a normal pro-$p$ subgroup of $G$ such that $G/ S \simeq \mathbb{Z}_p$, $G_0 \subseteq S$ and $S$ is of type $FP_m$ for some $m \geq 1$. Then there is  a normal pro-$p$ subgroup $S_0$ of $G$ such that $G/ S_0 \simeq \mathbb{Z}_p$, $S \not= S_0$, $G_0 \subseteq S_0$ and $S_0$ is of type $FP_m$.
\end{proposition}

\begin{proof} Note that since $S$ is a pro-$p$ group of type $FP_m$ and $G/ S\simeq \mathbb{Z}_p$ is a pro-$p$ group of type $FP_{\infty}$, hence of type $FP_m$, we can conclude that $G$ is a pro-$p$ group of type $FP_m$.
Set  $$Q = G/ G_0 = \langle x, y \rangle, \hbox{ where }H = S/ G_0 = \langle x \rangle.$$
Since $Q = G/ G_0$ is a finitely generated abelian pro-$p$ group and $G$ is of type $FP_m$ we conclude that $A_i = H_i(G_0, \mathbb{Z}_p)$ is finitely generated as a pro-$p$ $\mathbb{Z}_p[[Q]]$-module for $i \leq m$. Since $S$ is a pro-$p$ group of type $FP_m$ we conclude that $A_i$ is finitely generated as a pro-$p$ $\mathbb{Z}_p[[H]]$-module. Then by Propositionn \ref{King-used} for sufficently big $j$ we have that $A_i$ is finitely generated as a pro-$p$ $\mathbb{Z}_p[[H_j]]$-module, where $H_j = \langle x y^{ - p^j} \rangle \leq Q$, for every $ i \leq m$. Then we define $S_0$ as the preimage in $G$ of one such $H_j$.
\end{proof}

{\bf Proofs of Theorem \ref{Main1}}

There is a commutative diagram where the lines are short exact sequences of pro-$p$ groups
$$\xymatrix{K \ \ar@{^{(}->}[r] \ar[d]_{id_K} & \Pi \ar@{->>}[r]^{} \ar@{->>}[d]_{\pi} & F_n \ar@{->>}[d]^{} \\ K \ \ar@{^{(}->}[r] & G \ar@{->>}[r]^{} & \Gamma}$$
where $F_n$ is the free pro-$p$ group with a free basis $s_1, \ldots, s_n$. Define $$\Pi = \Pi_1 \coprod_K \Pi_2 \coprod_K \ldots \coprod_K \Pi_n,$$ where $\coprod_K$ is the amalgamated free product in the category of pro-$p$ groups, and each $\Pi_i = K \rtimes \langle s_i \rangle$, $\langle s_i \rangle \simeq \mathbb{Z}_p$. Note that since $K$ is normal in $\Pi$ and $\Pi/ K \simeq \Pi_1/ K \coprod \Pi_2/ K \coprod \ldots \coprod \Pi_n/ K$ is a free pro-$p$ product we conclude that $\Pi_1 \coprod_K \Pi_2 \coprod_K \ldots \coprod_K \Pi_i$ embeds in $\Pi$ for every $1 \leq i \leq n$.

Recall that $\Gamma^{ab}$ is infinite, hence  the image in $\Gamma^{ab}$ of at least one $\pi(s_i)$ has infinite order.   Without loss of generality we can assume that the image of $\pi(s_1)$ in $\Gamma^{ab}$ has infinite order. In particular $\Pi_1 \simeq \pi ({\Pi_1})$ is an isomorphism. Note that  $[K, s_1] \subseteq G' \cap K \subseteq  N$, hence $\Pi_1'\subseteq N$. We have $N \subseteq K \subseteq \Pi_1$ where $K/ N \simeq \mathbb{Z}_p, \Pi_1 / K \simeq \mathbb{Z}_p$, this together with the inclusion $\Pi_1'\subseteq N$ implies that $\Pi_1 / N \simeq \mathbb{Z}_p^2$. 

By assumption $K$ is of type $FP_{n_0}$. By Proposition \ref{quotientZp} there is $S_0$ a normal pro-$p$ subgroup of $\Pi_1$ such that $N \subseteq S_0$, $S_0$ is of type $FP_{n_0}$, $S_0 \not= K$ and $\Pi_1/ S_0 \simeq \mathbb{Z}_p$.  

Recall that $\Pi_1 \simeq \pi(\Pi_1)$. 
Let $$\mu : G \to \mathbb{Z}_p$$ be a homomorphism of pro-$p$ groups such that $Ker (\mu \circ \pi) \cap \Pi_1 = S_0$ i.e. $Ker(\mu) \cap \pi( \Pi_1) = \pi(S_0)$. This is possible since $\Pi_1 / N \simeq \mathbb{Z}_p^2$ is abelian and $G' \cap K \subseteq N \subseteq S_0$. Note that $K \not\subseteq S_0$, hence  $\mu(K) \not= 0$.

Consider the epimorphism of pro-$p$ groups $$\chi = \mu  \circ \pi : \Pi \to \mathbb{Z}_p.$$ Note that $\chi(K) \not= 0$, $Ker(\chi) \cap \Pi_1 = S_0 $ is of type $FP_{n_0}$ and $Ker(\chi) \cap K = S_0 \cap K = N$ is of type $FP_{n_0 - 1}$. Then we view $\Pi_1 \coprod_K \Pi_2$ as a proper HNN extension $$\langle \Pi_1, s_2 \ | \ K^{s_2} = K \rangle$$ with a pro-$p$ base group $\Pi_1$, associated pro-$p$ subgroup $K$ and stable letter $s_2$. Then by Lemma \ref{HNN} a) $$Ker(\chi) \cap ( \Pi_1 \coprod_K \Pi_2)\hbox{ is of type }FP_{n_0}.$$ We view $\Pi_1 \coprod_K \Pi_2 \coprod_K \Pi_3$ as a proper HNN extension with a base pro-$p$ group $\Pi_1 \coprod_K \Pi_2$, associated pro-$p$ subgroup $K$ and stable letter $s_3$ then by Lemma \ref{HNN} a) $$Ker(\chi) \cap ( \Pi_1 \coprod_K \Pi_2 \coprod_K \Pi_3)\hbox{ is of type }FP_{n_0}.$$ Then repeating this argument several times we deduce that $Ker(\chi)$ is of type $FP_{n_0}$.

By construction $Ker(\mu)$ is a quotient of $Ker(\chi)$. If $n_0 = 1$ then $Ker(\chi)$ is finitely generated (as a pro-$p$ group), then any pro-$p$ quotient of $Ker(\chi)$ is finitely generated (as a pro-$p$ group). In particular $Ker(\mu)$ is finitely generated (as a pro-$p$ group).

Now for the general case  i.e. $n_0 \geq 2$ we will apply Lemma \ref{L1}.  
Write $\widetilde{Ker (\chi)}$ for the image of $Ker (\chi)$ in $F_n$ and $\widetilde{Ker(\mu)}$ for the image of $Ker(\mu)$ in $\Gamma$. By construction $Ker(\chi) \cap K = N= Ker(\mu) \cap K$. By assumption $N$ is of type $FP_{n_0-1}$ and we have already shown that $Ker(\chi)$ is of type $FP_{n_0}$.  By construction $\mu(K) \not= 0$, hence $K. Ker (\mu) \not= Ker(\mu)$ and since $G/ Ker(\mu) \simeq \mathbb{Z}_p$ we deduce that $K. Ker(\mu)$ has finite index in $G$ and so $\widetilde{Ker(\mu)}$ has finite index in $\Gamma$. Since in the short exact sequence of pro-$p$ groups $$1 \to K \to G \to \Gamma \to 1$$ we have that $G$ and $K$ are pro-$p$ groups of type $FP_{n_0}$ (it suffices that $K$ is of type $FP_{n_0-1}$) we deduce that $\Gamma$ is of type $FP_{n_0}$. Then $\widetilde{Ker(\mu)}$ is a pro-$p$ group of type $FP_{n_0}$. Then we can apply Lemma \ref{L1} for the commutative diagram 

$$\xymatrix{N = Ker(\chi) \cap K \ \ar@{^{(}->}[r] \ar[d]_{id_N} &  Ker (\chi) \ar@{->>}[r]^{} \ar@{->>}[d]_{\pi |_{Ker(\chi)}} &  \widetilde{Ker(\chi)} \ar@{->>}[d]^{} \\ N = Ker(\mu) \cap K  \ \ar@{^{(}->}[r] & Ker (\mu) \ar@{->>}[r]^{} & \widetilde{Ker(\mu)}}$$
 to deduce that $Ker(\mu)$ is a pro-$p$ group of type $FP_{n_0}$. Finally we set $M = Ker(\mu)$.

\medskip
{\bf Proof of Corollary \ref{cor1}}

 We define $M$ as in the proof of Theorem \ref{Main1} for $\Gamma = F_n$ and $\pi$ the identity map, $\mu = \chi$. Thus $M = Ker(\chi) = Ker(\mu)$ is a normal subgroup of $G$, $G/ M \simeq \mathbb{Z}_p$ and $M$ is of type $FP_{n_0}$. We view $$G = \Pi = \Pi_1 \coprod_K \Pi_2 \coprod_K \ldots \coprod_K \Pi_n$$ as a proper HNN extension with a base pro-$p$ subgroup 
$A =\Pi_1 \coprod_K \Pi_2 \coprod_K \ldots \coprod_K \Pi_{n-1}$, associated pro-$p$ subgroup $K$ and stable letter $s_n$. By the proof of Theorem \ref{Main1} $A \cap M = A \cap  Ker (\chi)$ is of type $FP_{n_0}$. Suppose that $M$ is of type $FP_{n_0+1}$. By Lemma \ref{HNN} b) $N = M \cap K$ is of type $FP_{n_0}$, a contradiction. Hence $M$ is not of type $FP_{n_0 + 1}$. This completes the proof of the corollary.

{\bf Proof of Theorem \ref{ThmMain2}} 

We claim that there is a  finitely generated non-procyclic pro-$p$ subgroup $\Gamma_0$ of $\Gamma$ such that $\Gamma_0$ acts trivially on the abilianization  $K^{ab} = K/ K'$ via conjugation. Let $T = tor(K/ K')$ be the torsion part of $K^{ab}$. Then $V = K^{ab}/ T \simeq \mathbb{Z}_p^d$, where $d \geq 1$. Note that the conjugation action of $\Gamma$ on $V \simeq \mathbb{Z}_p^d$ induces a homomorphism $$\rho : \Gamma \to GL_d(\mathbb{Z}_p).$$ 
Note that $Im (\rho)$ is a pro-$p$ subgroup of $GL_d(\mathbb{Z}_p)$, hence is $p$-adic analytic and there is an upper bound on the number of generators of any finitely generated pro-$p$ subgroup of $Im (\rho)$ \cite{{DSMS}}. Hence $\rho$ is not injective. Alternatively we can use the main result of \cite{B-L} to deduce that 
 $\rho$ is not injective.  Thus $Ker(\rho)$ is a non-trivial normal pro-$p$ subgroup of $\Gamma$ and we can  choose $\Gamma_0$ any non-procyclic finitely generated pro-$p$ subgroup of $Ker(\rho)$.

Set $G_0 = K \rtimes \Gamma_0$. Then by Corollary \ref{cor1} there is a normal pro-$p$ subgroup $M$ of $G_0$ such that $G_0/ M \simeq \mathbb{Z}_p$ and $M$ is not of type $FP_2$ i.e. is not finitely presented as a pro-$p$ group. Thus $G_0$ is incoherent (in the category of  pro-$p$ groups). This completes the proof.

\medskip
We recall the definition of the class of pro-$p$ groups $\mathcal{L}$. It uses the extension of centraliser construction.
We define inductively the class $\mathcal{G}_n$ of pro-$p$ groups by setting
$\mathcal{G}_0$ as the class of all finitely generated free pro-$p$ groups and a group $G_{n} \in \mathcal{G}_n$ if there is a decomposition $G_n = G_{n-1} \coprod_{C} A$, where $G_{n-1} \in \mathcal{G}_{n-1}$, $C$ is self-centralised procyclic subgroup of $G_{n-1}$ and $A$ is a  finitely generated free abelian pro-$p$ group such that $C$ is a direct summand of $A$. The class $\mathcal{L}$ is defined as the class of all finitely generated pro-$p$ subgroups $G$ of $G_n$ where $G_n \in {\mathcal G}_n$ for $n \geq 0$. The minimal $n$ such that $G \leq G_n \in {\mathcal{G}_n}$ is called the weight of $G$.

\begin{proposition} \label{abel}
Let $K \in  {\mathcal L}$ be a non-trivial pro-$p$ group. Then $K^{ab} = K/ K'$ is infinite.
\end{proposition}

\begin{proof} Let $K \in {\mathcal{L}}$ have weight $n$. Suppose that $K^{ab}$ is finite. And $n$ is the smallest possible with $K^{ab}$ finite.
By \cite[Thm. B]{S-Z} $K$ is the fundamental pro-$p$ group of a finite graph of pro-$p$ groups $\Delta$, where each edge group is trivial or $\mathbb{Z}_p$ and each vertex groups is either a non-abelian limit pro-$p$ group of weight at most $n-1$ or a finitely generated  abelian pro-$p$ group. 

Let $\Gamma$ be the underlying graph of the finite graph of groups $\Delta$. If it is not a tree then $K$ decomposes as a pro-$p$ HNN extension, hence the stable letter generates an infinite procyclic subgroup of $K^{ab}$, a contradiction. 

 We can assume that $| V(\Gamma)|$ is the smallest possible.  Then we have a decomposition as an amalgamated pro-$p$ free product 
$K = K_0 \coprod_{G_{e_0}} G_{v_0}$, where $K_0$ is the fundamental pro-$p$ group of the subgraph of pro-$p$ groups $\Delta_0$ of $\Delta$ such that its underlying graph $\Gamma_0$ is obtained from $\Gamma$ by removing the edge $e_0$  and its vertex $v_0$ and we have that $e_0$ is the unique edge in $\Gamma$ that has $v_0$ as a vertex. Note that by \cite{Ri} every  amalgamated free pro-$p$ product with procyclic amalgamation is proper. Since the class $\mathcal{L}$ is closed under finitely generated pro-$p$ subgroups, $K_0 \in \mathcal{L}$ and by the minimality  of  $| V(\Gamma)|$ and $n$  we have that $K_0^{ab}$ and $G_v^{ab}$ are  infinite. If we write $t(M)$ for the torsion free rank of an abelian finitely generated pro-$p$ group $M$ then
$t(K^{ab}) \geq t(K_0^{ab}) + t(G_{v_0}^{ab}) - t(G_{e_0}) \geq 1 + 1 - t(G_{e_0}) \geq 1$, so $K^{ab}$ cannot be finite.
\end{proof}

{\bf Proof of Corollary \ref{CorMain*}}  By Proposition \ref{abel}  $K^{ab}$ is infinite. Let $N$ be a normal pro-$p$ subgroup of $K$ such that $K/ N \simeq \mathbb{Z}_p$. By part (4) from the main theorem of \cite{K-Z} we have that $N$ is not finitely generated as a pro-$p$ group. Then we can apply  Theorem  \ref{ThmMain2}.

\medskip

{\bf Proof of Corollary \ref{aut}} Let $F$ be a finitely generated free  non-procyclic pro-$p$ group that embeds as a closed subgroup of $Out(K)$.
Note that  $G = K \rtimes F$ is a pro-$p$ group embeds as a closed subgroup of $Aut(K)$ and by Theorem \ref{ThmMain2} $G$ is incoherent (in the category of pro-$p$ groups).

{\bf Proof of Corollary \ref{alternative}} 
We recall first  some results from \cite{L}. 
Let $G$ be a finitely generated pro-$p$ group and $Aut(G)$ denote all continuous automorphisms of $G$ (which coincide with the abstract automorphisms of $G$). Denote $Inn(G)$ the group of the internal automorphisms. The group $Aut(G)$ is a profinite group.

\begin{lemma} \cite{L} a) Let $G$ be a finitely generated pro-$p$ group and $G^*$ be the Frattini subgroup of $G$ i.e. the intersection of all maximal open subgroups of $G$. Then 
$Ker(Aut(G) \to Aut(G/ G^*))$ is a pro-$p$ subgroup of $Aut(G)$ of finite index.

b) Let $F$ be a finitely generated free pro-$p$ group and $N$ be a characteristic pro-$p$ subgroup of $F$. Then the map $Aut(F) \to Aut(F/N)$, obtained by taking the induced automorphisms, is surjective.
\end{lemma}

We set $Aut_0(G) = Ker(Aut(G) \to Aut(G/ G^*))$ and $Out_0(G) =Aut_0(G)/ Inn(G)$.

\begin{lemma} \label{final}
Suppose $K$ is a free pro-$p$ group, $d(K) = 2$ and $M $ is the maximal pro-$p$ metabelian quotient of $K$. Then $Out(M)$ contains a finitely generated pro-$p$ subgroup $H$   such that $H$ has a metabelian pro-$p$ quotient that is not finitely presented (as a pro-$p$ group).
\end{lemma}

{\bf Lemma \ref{final} implies Corollary \ref{alternative}:}
If $Out(K)$ contains a pro-$p$ free non-procyclic subgroup we can apply Corollary \ref{aut}. Then we can assume that $Out(K)$ does not contain a pro-$p$ free non-procyclic subgroup. We can further assume that the pro-$p$ version of the Bieri-Strebel result holds otherwise Corollary \ref{alternative} holds i.e. if a finitely presented pro-$p$ group does not contain a free non-procyclic pro-$p$ subgroup then any metabelian pro-$p$ quotient of that group is a finitely presented pro-$p$ group.

 Let $H$  be a pro-$p$ subgroup of $Out(M)$  as in Lemma \ref{final}. Since $Aut_0(M)$ has finite index  in $Aut(M)$ without loss of generality we can assume that $H \subseteq Out_0(M) $. The epimorphism of pro-$p$ groups $Aut_0(K) \to Aut_0(M)$ induces an epimorphism of pro-$p$ groups $Out_0(K) \to Out_0(M)$. Then there is a finitely generated pro-$p$ subgroup $\widetilde{H}$ of  $Out_0(K)$ that maps surjectively to $H$, in particular $\widetilde{H}$ has a metabelian pro-$p$ quotient that is not finitely presented (as a pro-$p$ group).
Then by the previous considerations $\widetilde{H}$ is not a finitely presented pro-$p$ group.

Note that  $Inn(K) \simeq K$. Consider the short exact sequence $1 \to K \to Aut_0(K) \to Out_0(K) \to 1$  and let $H_0$ be the preimage of $\widetilde{H}$ in $Aut_0(K)$. Then there is a short exact sequence $$ 1 \to K \to H_0 \to \widetilde{H} \to 1$$ of pro-$p$ groups. Since $K$ is a finitely generated pro-$p$ group we have that $H_0$ is a finitely generated pro-$p$ group and $H_0$ is not finitely presented otherwise $\widetilde{H}$ would be a finitely presented pro-$p$ group, a contradiction. Thus  $Aut_0(K)$ is incoherent (in the category of pro-$p$ groups).

{\bf Proof of Lemma \ref{final}} Here we use significantly ideas introduced in \cite{R0}.
We fix $x_1, x_2$ a generating set of $M$.
Define $$IAut(M) = \{ \varphi \in Aut(M) \ | \ \varphi \hbox{ induces on } M/ M'\hbox{ the identity map} \},$$ where $Aut(M)$ denotes continuous automorphisms of $M$. 
In fact every abstract automorphism of a finitely generated pro-$p$ group is a continous one.
Then there is a short exact sequence of profinite groups
$$1 \to IAut(M) \to Aut(M) \to Aut(M^{ ab}) = GL_2(\mathbb{Z}_p) \to 1.$$

By \cite{R0} there is a Bachmut embedding $\beta$ of $IAut(M)$ in $GL_2(\mathbb{Z}_p[[M^{ab}]])$, where $M^{ab}$ is the abelianization of $M$ i.e. the maximal pro-$p$ abelian quotient of $M$. By definition $$\beta(\varphi) = ( \partial (x_i^{ \varphi}) / \partial x_j),$$ where we use the notations from \cite{R0}, thus $Aut(M)$ in this proof acts on the right, $ \partial (x_i^{ \varphi}) / \partial x_j =  \partial / \partial x_j (x_i^{ \varphi}) $ and $$\partial/ \partial x_j : M \to \mathbb{Z}_p[[M^{ab}]]$$ are the Fox derivatives 
 defined by $$\partial / \partial x_j (1) = 0,   \partial / \partial x_j ( g_1 g_2) = \partial/ \partial x_j ( g_1) + \overline{g}_1 \partial / \partial x_j (g_2), \partial/ \partial x_j (x_i) = \delta_{i,j} \hbox{ the Kroniker symbol,}$$ where $\overline{g}_1$ is the image of $g_1 \in M$ in $ M^{ ab}$. Define  $det(\varphi) = det (\beta(\varphi))$. By \cite{R0} $$det(IAut(M)) = 1 + \Delta =: P$$ is a multiplicative abelian group, where $\Delta$ is the unique  maximal ideal of $\mathbb{Z}_p[[M^{ab}]]$, and the  $GL_2(\mathbb{Z}_p)$-action via conjugation on the abelianization of $IAut(M))$ induces an action on $det(IAut(M)) = P$. Then we have a short exact sequence of profinite groups $$1 \to P \to Aut(M)/ Ker (det) \to GL_2(\mathbb{Z}_p) \to 1.$$
 
 Consider the pro-$p$ group $$GL_2^1(\mathbb{Z}_p) = Ker(GL_2(\mathbb{Z}_p) \to GL_2(\mathbb{F}_p))$$ Let $Q$ be the maximal pro-$p$ quotient of $P$ that has exponent $p$. Then there is a pro-$p$ subgroup $T$ of $ Aut(M)/ Ker (det) $ and a short exact sequence of pro-$p$ groups $$ 1 \to P \to T \to GL_2^1(\mathbb{Z}_p) \to 1$$ and a pro-$p$ quotient $T_0$ of $T$ together with a short exact sequence of pro-$p$ groups $$ 1 \to Q \to T_0 \to 
  GL_2^1(\mathbb{Z}_p) \to 1.$$
  By \cite{R} $$P^p \cap ( 1 + p \Delta) = 1 + p^2 \Delta$$ and for $\delta \in \Delta$ using  $[\delta]$ for the image of $ 1 + p \delta$ in $Q$ we have that $$[\delta_1] [\delta_2] = [ \delta_1 + \delta_2].$$ Thus the multiplicative  subgroup of $Q$ generated by $\{ [\delta] \ | \ \delta \in \Delta \}$ could be identified with the additive group that is the image of $\Delta$ mod $p$ i.e. with the augmentation ideal $s_1 \mathbb{F}_p[[s_1, s_2]] + s_2 \mathbb{F}_p[[s_1, s_2]]$ of $\mathbb{F}_p[[s_1, s_2]]$, where $s_i$ is the image of $x_i-1$ in $\mathbb{Z}_p[[M^{ ab}]]$.

 Consider now $\varphi_2 \in Aut(M)$ given by
  $$\varphi_2 = \rho^p, \hbox{ where }\rho (x_1) = x_1 x_2, \rho(x_2) = x_2$$ and $\varphi_1 \in IAut(M)$ such that $$det(\varphi_1) = 1 + p s_1.$$ Note that $\varphi_1$ is not uniquely determined and that the image of $\varphi_2$ in $GL_2(\mathbb{Z}_p)$ is in $GL_2^1(\mathbb{Z}_p)$. Hence the profinite subgroup $\Gamma$ of $Aut(M)$ generated by $\varphi_1, \varphi_2$ is in fact a pro-$p$ group.  Let $$\Gamma_0 = \langle \psi_1, \psi_2 \rangle$$ be the image of $\Gamma$ in $T_0$, where $\psi_i$ is the image of $\varphi_i$ in $T_0$. Thus $\Gamma_0$ is a pro-$p$ group.
 
 By \cite[Prop. 4.4]{R0}  for every $\varphi \in IAut(M)$ for $\varphi'= \rho^{ -1} \varphi \rho$, $h'= det (\beta(\varphi'))$ and $h = det( \beta(\varphi))$ we have that $h'$ is obtained from $h$ applying the substitution $s_1 \to s_1 + s_2 + s_1 s_2$.

 Recall that by construction $det (\beta(\varphi_1)) = 1 + p s_1$. Then the action of $\psi_2$ on $\psi_1 = [s_1]$  by conjugations is induced by applying the substitution $s_1 \to s_1 + s_2 + s_1 s_2$ exactly $p$-times, thus gives the substitution $s_1 \to (1 + s_1) ( 1 + s_2) ^p -1$. 
 Similarly the action of $\psi_2^k$ on  $\psi_1 = [s_1]$ by conjugation is induced by applying the substitution $s_1 \to s_1 + s_2 + s_1 s_2$ exactly $pk$-times, thus gives the substitution $s_1 \to (1 + s_1) ( 1 + s_2) ^{pk} -1$. As explained above we can move to additive notation and work in the augmentation ideal $s_1 \mathbb{F}_p[[s_1, s_2]] + s_2 \mathbb{F}_p[[s_1, s_2]]$ of $\mathbb{F}_p[[s_1, s_2]]$.  This implies that the normal pro-$p$ subgroup $A$ of $\Gamma_0$ generated by $\psi_1$ can be identified with an aditive subgroup of  $s_1 \mathbb{F}_p[[s_1, s_2]] + s_2 \mathbb{F}_p[[s_1, s_2]]$ that contains $(1 + s_1) ( 1 + s_2) ^{p k} -1$ for $k \geq 0$, in particular $A$ is infinite. 
 
 Note that $\Gamma_0 \simeq A \rtimes \mathbb{Z}_p$, where $\mathbb{Z}_p $ is generated by $\psi_2$. We view $A$ as a $\mathbb{F}_p[[t]]$-module via the conjugation action of $\psi_2 = 1 + t$. Furthermore $A$ is a pro-$p$ cyclic  $\mathbb{F}_p[[t]]$-module, with a generator $\psi_1$. Since every proper $\mathbb{F}_p[[t]]$-module quotient of $\mathbb{F}_p[[t]]$ is a finite additive group, we deduce that $A \simeq \mathbb{F}_p[[t]]$. 
 Then  by the example after Theorem \ref{KK} $\Gamma_0$ is not a finitely presented pro-$p$ group.

  Note that the image $W$ of $M \simeq Inn (M)$ in $T_0$ is inside $Q$ and since $M$ is a finitely generated pro-$p$ group and $Q$ is an abelian pro-$p$ group of finite exponent $p$ then $W$ and consequently $\Gamma_0 \cap W$ are finite. 
 Since $\Gamma_0 \cap W$ is finite $\Gamma_0/ (\Gamma_0 \cap W)$ is  not a finitely presented pro-$p$ group. Actually examining the structure of $\Gamma_0$ it is easy to see that any finite normal subgroup of $\Gamma_0$ is trivial, in particular $\Gamma_0 \cap W = 1$. Finally $\Gamma_0 \simeq \Gamma_0 / (\Gamma_0 \cap W)$ is a metabelian pro-$p$ quotient of a 2-generated pro-$p$ group  $ H \leq Out(M)$. This completes the proof of the lemma.

\end{document}